\documentclass[a4paper,12pt]{article}
\usepackage[utf8]{inputenc}
\usepackage{amsmath,amsthm,amsfonts,latexsym,amssymb,bm}

\newtheorem{Thm}{Theorem}[section]

\newtheorem{Lem}[Thm]{Lemma}

\newtheorem{Prop}[Thm]{Proposition}

\newtheorem{Rem}[Thm]{Remark}

\newcommand{\R}{\mathbb{R}}
\newcommand{\Z}{\mathbb{Z}}
\newcommand{\T}{\mathbb{T}}
\newcommand{\N}{\mathbb{N}}
\newcommand{\eps}{\varepsilon}
\newcommand{\weak}{\rightharpoonup}
\renewcommand{\dot}{\breve}
\begin{document}
\begin{center}
{\Large\bf On the Fu\v{c}ik spectrum of the wave operator and an
 asymptotically linear problem}\\
\ \\
Pedro Girão\footnote{Email: pgirao@math.ist.utl.pt. 
Partially supported by the Fundação para a Ciência e a Tecnologia (Portugal) and by project
UTAustin/MAT/0035/2008.}
and Hossein Tehrani\footnote{Email:
tehranih@unlv.nevada.edu. This work was initiated while the author
was visiting IST Lisbon on a sabbatical from UNLV. The support of
both institutions is gratefully acknowledged.}
\\

\vspace{2.2mm}

{\small 
Center for Mathematical Analysis, Geometry and Dynamical Systems,\\
Instituto Superior Técnico,\\ Av.\ Rovisco Pais,
1049-001 Lisbon, Portugal\\
and\\
Department of Mathematical Sciences,\\
University of Nevada,\\
Las Vegas, NV 89154-4020, USA}
\end{center}

\begin{center}
{\bf Abstract}\\
\end{center}
\noindent We study generalized solutions of the nonlinear wave
equation
$$
u_{tt}-u_{ss}=au^+-bu^-+p(s,t,u),
$$
with periodic conditions in $t$ and homogeneous Dirichlet
conditions in $s$, under the assumption that the ratio
of the period to the length of the interval is two. When $p\equiv 0$ and $\lambda$ is a nonzero
eigenvalue of the wave operator, we give a proof of the existence
of two families of curves (which may coincide) in the Fu\v{c}ik
spectrum intersecting at $(\lambda,\lambda)$. This result is known
for some classes of self-adjoint operators (which does not cover
the situation we consider here), but in a smaller region than
ours. Our approach is based on a dual variational formulation and
is also applicable to other operators, such as the Laplacian. In
addition, we prove an existence result for the nonhomogeneous
situation, when the pair $(a,b)$ is not `between' the Fu\v{c}ik
curves passing through $(\lambda,\lambda)\neq(0,0)$ and $p$ is a
continuous function, sublinear at infinity.

\section{Introduction}

Let $\Omega$ be an open bounded region in $\R^N$ and
$L:D(L)\subset L^2(\Omega)\rightarrow L^2(\Omega)$ a self-adjoint
operator. Consider the equation
$$Lu=\alpha u^{+}-\beta u^{-},\qquad u\in D(L)$$
where $u^+=\max\{u,0\}$ and $u^-=u^+-u$. The set of points
$(\alpha, \beta)$ for which the equation has a nonzero solution is
called the Fu\v{c}ik spectrum of $L$. Since the pioneering work of
Fu\v{c}ik \cite{F}, there has been a growing interest in the study
of the structure of this set. In particular, if $\lambda$ is an
eigenvalue of $L$, then clearly $(\lambda, \lambda)$ is on the
Fu\v{c}ik spectrum. Of special interest is the study of
connected components of the Fu\v{c}ik spectrum that meet at
$(\lambda,\lambda)$. There are a number of works concerning the
structure of the Fu\v{c}ik spectrum for operators such as the
Laplace operator or, more generally, a self-adjoint operator with a
compact resolvent. But if one turns to the wave operator or to
operators with noncompact resolvent the list becomes considerably
smaller.

Here we only mention  the paper \cite{BFS}. In this work
Ben-Naoum, Fabry and Smets apply a Lyapunov-Schmidt decomposition
together with some contraction mapping arguments to give a
description of the Fu\v{c}ik spectrum of an operator away from its
essential spectrum, provided that some non-degeneracy conditions
are satisfied. These conditions do not seem to hold for the wave
operator however.

Before describing the main ideas in our approach, we briefly
recall some of the work on the Fu\v{c}ik spectrum of a
self-adjoint operator with compact resolvent. In the one
dimensional case the Fu\v{c}ik spectrum was completely described
by Fu\v{c}ik in \cite{F}.
Some of our spectrum will arise from this one dimensional case
(see Section~\ref{deux}).
If the space dimension is greater than
one, the complete understanding of the Fu\v{c}ik spectrum has been
more difficult. Ambrosetti and Prodi \cite{AP} obtained two
nonincreasing curves when the nonlinearity crosses the first
eigenvalue. This result was extended by Gallou\"et and Kavian
\cite{GK} and Ruf \cite{R} to the case where the nonlinearity
crosses a higher eigenvalue, in the case that the eigenvalue is
simple. Các \cite{C} generalized the result for the Laplacian
operator with Dirichlet boundary conditions.
Abchir~\cite{A} extended the result to a self-adjoint
operator with compact resolvent whose spectrum
is not bounded below. If
$\lambda_k$ is the $k$-th eigenvalue of $L$, we note that the
existence of the curves passing through the point
$(\lambda_k,\lambda_k)$, is usually established in the square
$\lambda_{k-1}\leq a,b\leq\lambda_{k+1}$. We also mention that
Marino, Micheletti and Pistoia in \cite{MMP} proved a detailed
result regarding some curves belonging to the Fu\v{c}ik spectrum
of the Laplacian on a bounded domain with Dirichlet conditions.
See also \cite{MP} where a different characterization of the
curves is given.

Finally in \cite{FG}, de Figueiredo and Gossez prove the existence
of a first Fu\v{c}ik curve through $(\lambda_2,\lambda_2)$ which
extends to infinity. Their result is for a general elliptic
operator in divergence form, and they provide a variational
characterization of the curve. In fact, drawing a line of positive
slope from $(\lambda_1,\lambda_1)$, de Figueiredo and Gossez obtain
the first intersection point with the Fu\v{c}ik spectrum through
some constrained minimization of a Dirichlet type integral. This
work provided the initial inspiration for our paper. Indeed,
applying a similar variational characterization, but to a suitably
{\it shifted}\/ dual problem, we are able to prove the existence of
two continuous curves through any nonzero eigenvalue point
$(\lambda_k, \lambda_k)$ of the wave operator. These curves, in a
sense made precise below, are the extreme parts of the Fu\v{c}ik
spectrum inside a rectangle containing the eigenvalue point. It is
worth mentioning that our approach can not only handle other
operators with noncompact resolvent (such as the beam operator),
but can also be applied to classes of operators that have already
been studied in the works mentioned above, providing new proofs
for known results but in regions that are larger than the square
$[\lambda_{k-1},\lambda_{k+1}]^2$.

In \cite{McK3} Choi, McKenna and Romano also used a dual
variational formulation and the mountain pass lemma
to prove existence of multiple periodic solutions
of the semilinear vibrating string model problem.

After completion of this work we learned that Ne\v{c}esal
\cite{N1} had proposed a somewhat similar approach (with no
theoretical justification) to obtain a numerical algorithm to
explore parts of the Fu\v{c}ik spectrum of the wave and beam
operators. Since the kernel of the wave operator is infinite
dimensional, it is not a priori clear that the dual problem has a
solution. In fact, in order to prove convergence of the maximizing
sequences, we have to employ an appropriate second shift which is
absent in Ne\v{c}esal's description.

Fu\v{c}ik \cite{F} and Dancer \cite{D} were the first to recognize
the importance of the Fu\v{c}ik spectrum in the study of
semilinear boundary value problems with linear growth at infinity.
Theorems on existence of solutions of the non-homogeneous equation
either treat so called type-I regions (when the pair $(a,b)$ does
not lie between the two curves above through a
$(\lambda_k,\lambda_k)$), or treat so called type-II regions (when
the pair $(a,b)$ lies between the two curves above through a
$(\lambda_k,\lambda_k)$). There is a substantial amount of work
done for type-I regions in the case of a self-adjoint operator
with compact resolvent. Results for type-II regions
are proved in the paper \cite{BFS},
where more general self-adjoint operators are also considered, as
mentioned above.

Several papers study the non-homogeneous wave equation.
Using the same boundary conditions that we use,
in~\cite{McK2} McKenna considers the case where $a=b$,
both at resonance and nonresonance, by reduction to a Landes\-man-Lazer
problem.
In \cite{W} Willem, with periodicity conditions on both variables,
overcomes the fact that the kernel of the wave operator is infinite
dimensional by proving a Continuation Theorem. Under
appropriate conditions, he proves the existence of at least
one generalized periodic solution of
$$
u_{tt}-u_{ss}=au^+-bu^-+p,
$$
with $\lim_{|u|\to+\infty}{p(s,t,u)}/{u}=0$, but only for $a$, $b$
in a box $0<\mu\leq a,b\leq\nu$,
$(\mu,\mu)\neq(a,b)\neq(\nu,\nu)$. Here $\mu$, $\nu$ are two
consecutive elements of the spectrum of the wave operator.
In~\cite{McK1} McKenna, Redlinger and Walter prove existence and
multiplicity results for an asymptotically homogeneous hyperbolic
problem when the nonlinearity is monotone in $u$. They are
particularly interested in the situation when the nonlinearity
crosses several eigenvalues. They reduce the problem to one on a
subspace on which the linear operator has a compact inverse; then
they apply degree theory. In \cite{McK5} Lazer and McKenna obtain
at least two solutions of a wave equation when the nonlinearity
crosses the first eigenvalue and is monotone increasing. In our
work, using our dual variational approach, we will also consider
the nonhomogeneous equation. We prove existence of a weak solution
for the parameters $a,b$ in larger type-I regions of the plane,
under a geometric condition for the linearization with respect to
$u$ of the nonlinearity.

We remark that if one restricts to a space of solutions
with certain symmetries, then one can remove zero from the
spectrum and end up with a compact operator, avoiding the
difficulties one usually finds in this type of problems. This idea
is due independently to Coron~\cite{Coron} and Vejvoda~\cite{Vej}.
It has been used by McKenna in \cite{McK4} to investigate
nonlinear oscillations in a suspension bridge.

Finally, in \cite{McK9} McKenna obtains results
for some situations in which the ratio of the period to the length of
the interval is irrational, sometimes referred to as a small divisors
problem. Two important differences arise.
The wave operator is invertible, and each point of the spectrum
is an eigenvalue of infinite multiplicity.

The organization of this work is as follows. First, in Sections~2
and 3 we treat the positive-homogeneous equation. Indeed, in
Section~2 we give a dual variational formulation for the Fu\v{c}ik
curves, and in Section~3 we prove existence of solutions of the
dual problems. In Section~4 we present and prove an existence
result for the non-homogeneous equation.
Our main results are Theorems~\ref{curvau}, \ref{curvad}, \ref{nonu}
and \ref{nond}.

\section{Two equivalent problems}\label{deux}
We denote by $\T$ the circle $\R/(2\pi\Z)$. Consider the wave
operator defined on $\{u\in H^2([0,\pi]\times\T)\::\:
u(0,t)=u(\pi,t)=0\ \mbox{for}\ t\in\T
\}$. Its eigenvalues are given by $\lambda_{(m,n)}=m^2-n^2$ for
any $(m,n)\in\N\times\N_0$. The functions
$$
 \phi_{(m,n)}=\textstyle\frac{\sqrt{2}}{\pi}\sin(ms)\cos(nt),\qquad
 \psi_{(m,n)}=\textstyle\frac{\sqrt{2}}{\pi}\sin(ms)\sin(nt)
$$
are eigenfunctions associated to $\lambda_{(m,n)}$ (see \cite{M}).
We let ${\cal H}=L^2([0,\pi]\times\T)$,
\begin{equation}\label{alp}
{\cal
H}=\left\{v=\sum[\alpha_{(m,n)}\phi_{(m,n)}+\beta_{(m,n)}\psi_{(m,n)}]\::\:
\sum[\alpha_{(m,n)}^2+\beta_{(m,n)}^2]<\infty\right\},
\end{equation}
and define
$${\cal R}=\overline{\mbox{span}\{\phi_{(m,n)},\psi_{(m,n)}\ \mbox{with}\ m\in\N,n\in\N_0, m\neq n\}}
,$$
$${\cal N}=\overline{\mbox{span}\{\phi_{(l,l)},\psi_{(l,l)}\ \mbox{with}\ l\in\N\}}
,$$ where the closures are in ${\cal H}$, so that ${\cal H}={\cal
R}\oplus{\cal N}$. Let also ${\cal H}^1={\cal R}\cap H^1([0,\pi]\times\T)$. We
denote by $\ldots\lambda_{-2}<\lambda_{-1}<\lambda_0=0<
\lambda_1<\lambda_2<\ldots$ the eigenvalues ordered according to
their value. Suppose that $\lambda_k\neq 0$. We are concerned with
finding continuous functions $a$ and $b$ from $\R^+$ to $\R$,
satisfying $a(1)=b(1)=\lambda_k$, such that for each $r\in\R^+$
the problem
\begin{eqnarray}\label{z}
\Box u=u_{tt}-u_{ss}=a(r)u^+-b(r)u^-&&\mbox{in}\:]0,\pi[\times\T,\\
u(0,t)=u(\pi,t)=0\qquad\qquad\quad\ \ \ &&\mbox{for}\ 0\leq t\leq
2\pi,\label{bdr}
\end{eqnarray}
has a nonzero weak solution $u$. As before $u^+=\max\{u,0\}$ and
$u^-=u^+-u$. By a weak solution we mean a critical point of the
energy functional $\cal I$ , defined on ${\cal H}^1\times\cal N$;

$$
{\cal I}(x,y)=
{\textstyle\frac{1}{2}}\left(\|x_s\|^2-\|x_t\|^2-a\|(x+y)^+\|^2-b\|(x+y)^-\|^2\right).
$$
A word on notation. Throughout $\langle\ \cdot\ ,\ \cdot\ \rangle$
denotes the inner product in ${\cal H}$ and $\|\,\cdot\,\|$ the
corresponding norm.

We identify the pair $(x,y)$ with $u=x+y$. Since the
functions $\phi_{(m,n)}$ and $\psi_{(m,n)}$ satisfy (\ref{bdr}), the
trace theorem implies that any $x\in {\cal H}^1$ has $L^2$ trace
on $\{0\}\times\T\,\,\cup\,\{\pi\}\times\T$ satisfying
(\ref{bdr}). On the other hand, any function $y\in {\cal N}$ is of
the form $y(s,t)=\kappa(t+s)-\kappa(t-s)$ for some $\kappa\in
L^2(\T)$ and thus also satisfies (\ref{bdr}). Therefore, any
critical point of ${\cal I}$ satisfies the boundary conditions
(\ref{bdr}).

We should mention that some of the Fu\v{c}ik spectrum
will come from solutions of the form
$$u(s,t)=\sin s \times T(t)$$
with $T''+ T = a T^+ - b T^-$,
and from solutions of the form
$$
u(s,t)=S(s) \times 1
$$
with $- S'' = a S^+ - b S^-$. This produces Fu\v{c}ik curves
through the eigenvalues $1-n^2$ and $m^2$, respectively.
The curves can be obtained explicitly (see~\cite{F}).

\subsection{Variational formulation for the
Fu\v{c}ik curve through\\ $(\lambda_k,\lambda_k)$ closest to $(\lambda_{k-1},\lambda_{k-1})$}

Let $k$ be a nonzero integer.
Choose a small parameter $\eps_1$ with $0<\eps_1<\lambda_k-\lambda_{k-1}$.
Suppose $u$ is a weak solution of (\ref{z})-(\ref{bdr}). Then $u$ solves
\begin{equation}\label{u}
 (\Box-(\lambda_{k-1}+\eps_1))u=(a-\lambda_{k-1}-\eps_1)u^+-
(b-\lambda_{k-1}-\eps_1)u^-
\end{equation}
together with (\ref{bdr}).
We write $u$ as
\begin{equation}\label{tet}
u=\sum\left[\theta_{(m,n)}\phi_{(m,n)}+\omega_{(m,n)}\psi_{(m,n)}\right].
\end{equation}
Suppose $a$, $b>\lambda_{k-1}+\eps_1$. Consider the convex function $H:\R\to\R$
defined by
\begin{equation}\label{h}
H(u)=\frac{1}{2}(a-\lambda_{k-1}-\eps_1)(u^+)^2+\frac{1}{2}
(b-\lambda_{k-1}-\eps_1)(u^-)^2.
\end{equation}
Let
$$
v:=\nabla H(u)=(a-\lambda_{k-1}-\eps_1)u^+-
(b-\lambda_{k-1}-\eps_1)u^-.
$$ From (\ref{u}) we deduce that $u=Lv$ for $L:{\cal H}\to{\cal H}$
defined by
$$
Lv=\sum\left[{\textstyle\frac{\alpha_{(m,n)}}{m^2-n^2-\lambda_{k-1}-\eps_1}}\phi_{(m,n)}+
{\textstyle\frac{\beta_{(m,n)}}{m^2-n^2-\lambda_{k-1}-\eps_1}}\psi_{(m,n)}\right]
$$
and $v$ written as in (\ref{alp}).
Hence the function $v$
satisfies
\begin{equation}\label{v}
 Lv=\nabla H^*(v)=\frac{1}{a-\lambda_{k-1}-\eps_1}v^+-
\frac{1}{b-\lambda_{k-1}-\eps_1}v^-.
\end{equation}
where $H^{*}$ is the Fenchel-Legendre transform of the convex
function $H$.
\par\vspace{.5cm}
Conversely, suppose now $v\in{\cal H}$ is a solution of (\ref{v}).
We may decompose $Lv=x+y$ with $x\in{\cal R}$ and $y\in{\cal N}$. In fact, the next lemma shows $x\in{\cal H}^1$.
\begin{Lem}\label{infact}
Let $v$ belong to ${\cal H}$ and $x$ be the component of $Lv$ in ${\cal R}$. Then $x$ belongs to the space ${\cal H}^1$.
\end{Lem}
\begin{proof}
The function $x$ may be expanded as
$$
x=\sum_{|m-n|\geq 1}\left[{\textstyle\frac{\alpha_{(m,n)}}{m^2-n^2-\lambda_{k-1}-\eps_1}}\phi_{(m,n)}+
{\textstyle\frac{\beta_{(m,n)}}{m^2-n^2-\lambda_{k-1}-\eps_1}}\psi_{(m,n)}\right].
$$
Since $\lambda_{k-1}+\eps_1$ is not an eigenvalue of the wave operator,
there exists $\delta>0$ such that
\begin{equation}\label{delta}
|m^2-n^2-\lambda_{k-1}-\eps_1|>\delta
\end{equation}
for all $m\in\N$ and $n\in\N_0$. Without loss of generality, we assume that
$\lambda_{k-1}+\eps_1>0$. Otherwise interchange the roles of $m$ and $n$ below.
We claim that there exists $\eps>0$ such that $\bigl|m-\sqrt{n^2+\lambda_{k-1}+\eps_1}\bigr|>\eps$
for all $m\in\N$ and $n\in\N_0$ with $|m-n|\geq 1$.
We prove this assertion by contradiction.
Suppose there exist a sequence $(m_l,n_l)$ with
\begin{equation}\label{one}
|m_l-n_l|\geq 1
\end{equation}
and
$\bigl|m_l-\sqrt{n_l^2+\lambda_{k-1}+\eps_1}\bigr|\leq 1/l$. Either $m_l$ and $n_l$ are both
bounded, or both sequences are unbounded. In the former case, modulo a subsequence,
$m_l=m_0$ and $n_l=n_0$ for large $l$ and $m_0=\sqrt{n_0^2+\lambda_{k-1}+\eps_1}$.
This contradicts inequality (\ref{delta}). In the latter case,
\begin{eqnarray*}
|m_l-n_l|&\leq&\left|m_l-\sqrt{n_l^2+\lambda_{k-1}+\eps_1}\right|+
\left|\sqrt{n_l^2+\lambda_{k-1}+\eps_1}-n_l\right|\\
&\leq& \frac{1}{l}+
\frac{\lambda_{k-1}+\eps_1}{\sqrt{n_l^2+\lambda_{k-1}+\eps_1}+n_l}\longrightarrow
0\qquad\mbox{as}\ l\to\infty,
\end{eqnarray*}
which contradicts (\ref{one}). This completes the proof of the claim.
Hence
$$
\frac{m^2+n^2}{(m^2-n^2-\lambda_{k-1}-\eps_1)^2}\leq\frac{1}{\eps^2}
\frac{m^2+n^2}{(m+\sqrt{n^2+\lambda_{k-1}+\eps_1})^2}\leq\frac{1}{\eps^2}
$$
implying that $x\in{\cal H}^1$.
\end{proof}
Let $u=Lv=\nabla H^*(v)$. If we write $v$ as in (\ref{alp}) and
$u$ as in (\ref{tet}), then
$(m^2-n^2-\lambda_{k-1}-\eps_1)\theta_{(m,n)}=\alpha_{(m,n)}$ and
$(m^2-n^2-\lambda_{k-1}-\eps_1)\omega_{(m,n)}=\beta_{(m,n)}$.
Therefore $u$ is a weak solution of (\ref{u}) together with
(\ref{bdr}).

We now choose a small parameter $\eps_2$ with
$$0<\eps_2<\frac{1}{\lambda_k-\lambda_{k-1}-\eps_1}-\max\left\{0,-\frac{1}{\lambda_{k-1}+\eps_1}\right\}$$
and take
\begin{equation}\label{muu}
\mu=\max\left\{0,-\frac{1}{\lambda_{k-1}+\eps_1}\right\}+\eps_2.
\end{equation}
We rewrite (\ref{v}) as
\begin{equation}\label{vv}
 (L-\mu)v=\hat{a}v^+-\hat{b}v^-,
\end{equation}
with
\begin{equation}\label{abt}
\hat{a}=\frac{1}{a-\lambda_{k-1}-\eps_1}-\mu,\qquad
\hat{b}=\frac{1}{b-\lambda_{k-1}-\eps_1}-\mu.
\end{equation}
Note that the greatest eigenvalue of $L-\mu$
is $1/(\lambda_k-\lambda_{k-1}-\eps_1)-\mu>0$ and
the space ${\cal N}$ is associated to
the eigenvalue
$\nu:=-1/(\lambda_{k-1}+\eps_1)-\mu<0$.

Fixing $r>0$, we consider solutions of (\ref{vv}) on the line $\hat{b}=r\hat{a}$:
\begin{equation}\label{vvv}
 (L-\mu)v=\hat{a}(v^+-rv^-).
\end{equation}
Next we introduce the $C^1$ functionals $F$ and $G$ from ${\cal H}$ to $\R$ by
\begin{eqnarray*}
 F(v)&=&
\langle (L-\mu)v,v\rangle,\\
 G(v)&=&
\|v^+\|^2+r\|v^-\|^2.
\end{eqnarray*}
To find the point on the line $\hat{b}=r\hat{a}$ farthest away
from the origin for which (\ref{vvv}) has a nonzero solution we
consider the maximization problem
\begin{equation}\label{sup}
\sup_{G(v)=1}F(v)=\,\sup_{v\neq 0}\,\frac{F(v)}{G(v)},\qquad v\in \cal H.
\end{equation}
Finally  $\check{a}(r)$ denotes the value of this supremum.

\subsection{Variational formulation for the
Fu\v{c}ik curve through\\ $(\lambda_k,\lambda_k)$ closest to $(\lambda_{k+1},\lambda_{k+1})$}

Choose a small parameter $\eps_3$ with $0<\eps_3<\lambda_{k+1}-\lambda_k$.
If $u$ is a weak solution of (\ref{z})-(\ref{bdr}) then $u$ solves
\begin{equation}\label{ubis}
 (-\Box+\lambda_{k+1}-\eps_3)u=(\lambda_{k+1}-a-\eps_3)u^+-
(\lambda_{k+1}-b-\eps_3)u^-
\end{equation}
together with (\ref{bdr}).
Suppose $a$, $b<\lambda_{k+1}-\eps_3$.
Let 
$K:\R\to\R$ be the convex function
defined by
$$K(u)=\frac{1}{2}(\lambda_{k+1}-a-\eps_3)(u^+)^2+\frac{1}{2}
(\lambda_{k+1}-b-\eps_3)(u^-)^2,$$
and $M:{\cal H}\to{\cal H}$ be defined by
$$
Mv=\sum\left[{\textstyle\frac{\alpha_{(m,n)}}{\lambda_{k+1}-m^2+n^2-\eps_3}}\phi_{(m,n)}+
{\textstyle\frac{\beta_{(m,n)}}{\lambda_{k+1}-m^2+n^2-\eps_3}}\psi_{(m,n)}\right],
$$
for $v$ as in (\ref{alp}).
If $v=\nabla K(u)$, then
\begin{equation}\label{vbis}
 Mv=\nabla K^*(v)=\frac{1}{\lambda_{k+1}-a-\eps_3}v^+-
\frac{1}{\lambda_{k+1}-b-\eps_3}v^-.
\end{equation}
And conversely, any solution of (\ref{vbis}) is a weak solution of (\ref{ubis})
together with (\ref{bdr}).
Finally, we choose a small parameter $\eps_4$ with
$$0<\eps_4<\frac{1}{\lambda_{k+1}-\lambda_k-\eps_3}-\max\left\{0,\frac{1}{\lambda_{k+1}-\eps_3}\right\}$$
and take
$$
\rho=\max\left\{0,\frac{1}{\lambda_{k+1}-\eps_3}\right\}+\eps_4.
$$
We rewrite (\ref{vbis}) as
\begin{equation}\label{vvbis}
 (M-\rho)v=\bar{a}v^+-\bar{b}v^-,
\end{equation}
with
$$
\bar{a}=\frac{1}{\lambda_{k+1}-a-\eps_3}-\rho,\qquad
\bar{b}=\frac{1}{\lambda_{k+1}-b-\eps_3}-\rho.
$$
Note that the greatest eigenvalue of $M-\rho$
is $1/(\lambda_{k+1}-\lambda_k-\eps_3)-\rho>0$ and
the space ${\cal N}$ is associated to
the eigenvalue
$\sigma:=1/(\lambda_{k+1}-\eps_3)-\rho<0$.

Again fixing  $r>0$, we consider solutions of (\ref{vvbis})
on the line $\bar{b}=r\bar{a}$:
\begin{equation}\label{vvvbis}
 (M-\rho)v=\bar{a}(v^+-rv^-).
\end{equation}
Finally we consider
\begin{equation}\label{supup}
\sup_{v\neq 0}\,\frac{\langle(M-\rho)v,v\rangle}{G(v)},\qquad v\in \cal H
\end{equation}
and denote by
$\tilde{a}(r)$ the value of this supremum.

\section{Proof of existence of solutions of the dual problems}

In this section we prove the existence of maximizers for the
problems (\ref{sup}) and (\ref{supup}) and examine simple
properties of the maxima $\check{a}(r)$ and $\tilde{a}(r)$.
Translating these results in terms of the parameters of the
original equation we obtain the two curves in the Fu\v{c}ik
spectrum stated in Theorems~\ref{curvau} and \ref{curvad}.

\begin{Prop}\label{att}
 The supremum in\/ {\rm (\ref{sup})} is attained.
\end{Prop}
\begin{proof}
Let $v_n$ be a maximizing
sequence for (\ref{sup}) such that $G(v_n)=1$. We write $v_n=w_n+z_n$ with
$w_n\in{\cal R}$ and $z_n\in{\cal N}$. Modulo a subsequence, we may assume
that $v_n^+\weak \gamma$, $v_n^-\weak \eta$, $v_n\weak v_0$, $w_n\weak w_0$
and $z_n\weak z_0$, all in ${\cal H}$. We remark that $v_0=\gamma-\eta=v_0^+-v_0^-$
with $\gamma\geq v_0^+$ and $\eta\geq v_0^-$.
We prove that $\limsup\|w_n\|=\|w\|$ and $\limsup\|z_n\|=\|z\|$. This will imply that
$\lim\|w_n\|=\|w_0\|$ and $\lim\|z_n\|=\|z_0\|$, so that $v_n\to v$ in ${\cal H}$. To do so
we prove that for any subsequence along which both $\|w_n\|$ and $\|z_n\|$ converge, we have
$\lim\|w_n\|=\|w_0\|$ and $\lim\|z_n\|=\|z_0\|$. So suppose that $\|w_n\|$ and $\|z_n\|$
converge. Since $Lw_n\to Lw_0$ strongly in ${\cal H}$ (see Lemma~\ref{infact}), we have
\begin{eqnarray*}
 \check{a}(r)&=&\lim\langle(L-\mu)v_n,v_n\rangle\\
   &=&\lim\left[\langle(L-\mu)w_n,w_n\rangle+\langle(L-\mu)z_n,z_n\rangle\right]\\
   &=&\lim\left[\langle(L-\mu)w_0,w_0\rangle-\mu(\|w_n\|^2-\|w_0\|^2)\right.\\
    &&\qquad\left.+\langle(L-\mu)z_0,z_0\rangle+\nu(\|z_n\|^2-\|z_0\|^2)\right]\\
    &=&\langle(L-\mu)v_0,v_0\rangle-\mu\lim(\|w_n\|^2-\|w_0\|^2)
    +\nu\lim(\|z_n\|^2-\|z_0\|^2).
\end{eqnarray*}
Recalling that $\mu>0$ and $\nu<0$, we conclude
\begin{equation}\label{pos}
F(v_0)=\langle(L-\mu)v_0,v_0\rangle\geq \check{a}(r),
\end{equation}
and the inequality is strict, unless $\lim\|w_n\|=\|w_0\|$ and $\lim\|z_n\|=\|z_0\|$.
On the other hand from
$$G(v_n)=
\|v^+_n\|^2+r\|v^-_n\|^2 =1$$ it follows that
\begin{equation}\label{g}
 G(v_0)=
\|v^+_0\|^2+r\|v^-_0\|^2 \leq
\|\gamma\|^2+r\|\eta\|^2 \leq 1.
\end{equation}
We know that $\check{a}(r)>0$ by testing $F/G$ with an
eigenfunction associated to $\lambda_k$. Thus inequality
(\ref{pos}) shows that $v_0\neq 0$. Clearly,
$F(v_0)/G(v_0)\leq\check{a}(r)$, so (\ref{pos}) and (\ref{g})
combined imply that $G(v_0)=1$,  $F(v_0)=\check{a}(r)$ and the
sequence $v_n$ converges to $v_0$ strongly in ${\cal H}$.
\end{proof}
\begin{Rem}
The purpose of the second shift $-\mu I$ is twofold. First, to
guarantee that the eigenvalue $\nu$ of $L-\mu$ associated to
${\cal N}$ is negative. And second, given a maximizing sequence,
to guarantee the convergence of the sequence of components in the
range of the wave operator.
\end{Rem}

We may write
\begin{equation}\label{checka}
\check{a}(r)=F(v_0)/G(v_0).
\end{equation}
For all $\varphi\in{\cal H}$ we have
$$
\left.\left(\frac{F}{G}\right)^\prime\right|_{v=v_0}\varphi=
\frac{F'(v_0)-\check{a}(r)G'(v_0)}{G(v_0)}\,\varphi=0,
$$
so that
$$
F'(v_0)=\check{a}(r)G'(v_0).
$$
We conclude that $v_0$ is a nonzero solution of (\ref{vvv}) with
$\hat{a}=\check{a}(r)$.  And there can be no nontrivial solution
of (\ref{vvv}) with $\hat{a}>\check{a}(r)$ for otherwise the functional
$F/G$ would have a critical value greater than its maximum.

Observe that the function $\check{a}$ is strictly decreasing unless the maximizer $v_0$ has
a fixed sign. If so, the function $v_0$ is an eigenfunction and it must correspond to $\lambda_{(1,0)}=1$.
\begin{Lem}\label{cont}
 The function $\check{a}$ is continuous.
\end{Lem}
\begin{proof}
Let $v_0(r)$ be such that
$\check{a}(r)=F(v_0(r))$ with $G_r(v_0(r))=1$. Let $r_0>0$.
The inequality
$$
\check{a}(r_0)\geq\frac{F(v_0(r))}{G_{r_0}(v_0(r))}=
\frac{\check{a}(r)}{G_r(v_0(r))+(r_0-r)\|v_0^-(r)\|^2}
$$
implies
$$
\limsup_{r\to r_0}\check{a}(r)\leq\check{a}(r_0).
$$
On the other hand, the inequality
$$
\check{a}(r)\geq\frac{F(v_0(r_0))}{G_{r}(v_0(r_0))}=
\frac{\check{a}(r_0)}{G_{r_0}(v_0(r_0))+(r-r_0)\|v_0^-(r_0)\|^2}
$$
implies
$$
\liminf_{r\to r_0}\check{a}(r)\geq\check{a}(r_0).
$$
This proves the continuity of $\check{a}$ at $r_0$.
\end{proof}
Let us now explicitly indicate the dependence of $G$ on $r$ by writing $G_r$.
The equality
$$
\frac{F(-v)}{G_{\frac{1}{r}}(-v)}=r\frac{F(v)}{G_r(v)},
$$
valid for all $v\neq 0$, implies that
\begin{equation}\textstyle
\check{a}\left(\frac{1}{r}\right)=r\check{a}(r).\label{sym}
\end{equation}
So
\begin{equation}\textstyle
\left(\check{a}\left(\frac{1}{r}\right),\frac{1}{r}\check{a}\left(\frac{1}{r}\right)\right)=
\left(r\check{a}(r),\check{a}(r)\right).\label{symm}
\end{equation}

We can now state
\begin{Thm}\label{curvau}
 Fix $\lambda_k\neq 0$. If $\lambda_k>0$ let $\check{Q}_k=]\lambda_{k-1},+\infty[^2$,
and if $\lambda_k<0$ let $\check{Q}_k=]\lambda_{k-1},0[^2$. There exists a continuous
curve ${\cal C}_k\subset\check{Q}_k$ through $(\lambda_k,\lambda_k)$ (symmetric
with respect to the line that bisects the odd quadrants and, if $k\neq 1$,
the graph of a strictly decreasing
function) such that $(a,b)\in{\cal C}_k$ implies {\rm (\ref{z})-(\ref{bdr})} has a nontrivial
weak solution. If $(a,b)$ lies in $\check{Q}_k$ and\/ {\rm below} ${\cal C}_k$, then
{\rm (\ref{z})-(\ref{bdr})} has no nontrivial weak solution.
\end{Thm}
\begin{proof}
Let $0<r<\infty$ be fixed and
consider the half-line $\hat{b}=r\hat{a}$ with $\hat{a}>0$
in the $(\hat{a},\hat{b})$-plane. We have seen that there
are no nontrivial solutions of (\ref{vv}) on this line with $\hat{a}>\check{a}(r)$.
We can write the original parameters $a$ and $b$ in (\ref{z}) in terms of $\hat{a}$ and
$\hat{b}$ as
\begin{equation}\label{sstar}
a=\lambda_{k-1}+\eps_1+\frac{1}{\hat{a}+\mu},\qquad b=\lambda_{k-1}+\eps_1+\frac{1}{\hat{b}+\mu}.
\end{equation}
As we increase $\hat{a}$ from zero to infinity, the image of the
half-line in the $(a,b)$-plane is a curve starting at
$(\lambda_{k-1}+\eps_1+1/\mu,\lambda_{k-1}+\eps_1+1/\mu)$ and
ending at $(\lambda_{k-1}+\eps_1,\lambda_{k-1}+\eps_1)$.
The map given by (\ref{sstar}) is a bijection between
$]0,+\infty[^2$ in the $(\hat{a},\hat{b})$ plane
and
\begin{equation}\label{qk}
Q_k:=\ ]p_k,q_k[^2\, :=\,
\left]\lambda_{k-1}+\eps_1,\lambda_{k-1}+\eps_1+{1}/{\mu}\right[^2
\end{equation}
in the $(a,b)$ plane.
If $\lambda_k>0$ then
$Q_k=\left]\lambda_{k-1}+\eps_1,\lambda_{k-1}+\eps_1+{1}/{\eps_2}\right[^2$,
and if $\lambda_k<0$ then
$$Q_k=\left]\lambda_{k-1}+\eps_1,-
\eps_2\frac{(\lambda_{k-1}+\eps_1)^2}{1-\eps_2(\lambda_{k-1}+\eps_1)}\right[^2.$$

Consider the curve ${\cal C}_k$ parametrized by
$$
r\longmapsto
\left(\lambda_{k-1}+\eps_1+\frac{1}{\check{a}(r)+\mu},
\lambda_{k-1}+\eps_1+\frac{1}{r\check{a}(r)+\mu}\right):=(a(r),b(r)).
$$
Clearly, $\check{a}(1)$ is the greatest eigenvalue of $L-\mu$, so the curve
${\cal C}_k$ passes through $(\lambda_k,\lambda_k)$.
As $r\to+\infty$, either $b(r)\to\lambda_{k-1}+\eps_1$ (if $\lim_{r\to+\infty}\check{a}(r)>0$),
or $a(r)\to\lambda_{k-1}+\eps_1+1/\mu$ (if $\lim_{r\to+\infty}\check{a}(r)=0$), or
both. On the other end, as $r\to 0$, either $a(r)\to\lambda_{k-1}+\eps_1$ (if $\lim_{r\to 0}\check{a}(r)=+\infty$),
or $b(r)\to\lambda_{k-1}+\eps_1+1/\mu$ (if $\lim_{r\to 0}\check{a}(r)<+\infty$), or both.
Therefore ${\cal C}_k$ approaches the boundary of $Q_k$ as $r\to 0$ and
$r\to+\infty$.

   From (\ref{sym}), note also that $(\check{a}(r),r\check{a}(r))=
\left(\check{a}(r),\check{a}\left(\frac{1}{r}\right)\right)$.
Suppose $\lambda_k\neq 1$. As $r$ increases, $\check{a}(r)$
decreases and $\check{a}\left(\frac{1}{r}\right)$ increases. The
curve ${\cal C}_k$ is the graph of a strictly decreasing function
$b=b(a)$. It lies in $\{(a,b)\in\R^2\::\:b<\lambda_k<a\ \mbox{or}\
a<\lambda_k<b\ \mbox{or}\ a=\lambda_k=b\}$. In addition
(\ref{symm}) implies that the curve ${\cal C}_k$ is symmetric with
respect to the line $b=a$.

Lemma~\ref{cont} implies that ${\cal C}_k$ is continuous. The
curve ${\cal C}_k$  divides the square $Q_k$ into two connected
components. We say that a point is inside $Q_k$ and {\em below}\/
${\cal C}_k$ if it belongs to the component which contains
$]\lambda_{k-1}+\eps_1,\lambda_k[^2$. There are no points on the
Fu\v{c}ik spectrum of (\ref{z})-(\ref{bdr}) inside the square
$Q_k$ and below ${\cal C}_k$.

Finally we explicitly indicate the dependence of the square $Q_k$
on $\eps_1$ and $\eps_2$ by writing $Q_k=Q_k(\eps_1,\eps_2)$.
If $\lambda_k>0$ we define $\check{Q}_k
:=]\lambda_{k-1},+\infty[^2$, and if $\lambda_k<0$ we define
$\check{Q}_k
:=]\lambda_{k-1},0[^2$. If $(a,b)\in\check{Q}_k$ then we can choose
$\eps_1$, $\eps_2$ sufficiently small so that $(a,b)\in Q_k(\eps_1,\eps_2)$.

Note that if we consider two squares $Q_k$, corresponding to two choices
of the pair $(\eps_1,\eps_2)$, then in their intersection the two curves ${\cal C}_k$
above coincide, for the points of ${\cal C}_k$ belong to the Fu\v{c}ik spectrum of (\ref{z})-(\ref{bdr}),
and the points in the intersection of the squares and below any one of the two curves
${\cal C}_k$ do not belong to the Fu\v{c}ik spectrum of (\ref{z})-(\ref{bdr}).
\end{proof}

A similar proof shows the supremum in (\ref{supup})
is attained and there are no nontrivial solutions
of (\ref{vvvbis}) with $\bar{a}>\tilde{a}(r)$. This leads to
\begin{Thm}\label{curvad}
 Fix $\lambda_k\neq 0$. If $\lambda_k<0$ let $\tilde{R}_k=]-\infty,\lambda_{k+1}[^2$,
and if $\lambda_k>0$ let $\tilde{R}_k=]0,\lambda_{k+1}[^2$. There exists a continuous
curve ${\cal D}_k\subset\tilde{R}_k$ through $(\lambda_k,\lambda_k)$ (symmetric
with respect to the line that bisects the odd quadrants and, if $k\neq 1$,
the graph of a strictly decreasing
function) such that $(a,b)\in{\cal D}_k$ implies {\rm (\ref{z})-(\ref{bdr})} has a nontrivial
weak solution. If $(a,b)$ lies in $\tilde{R}_k$ and\/ {\rm above} ${\cal D}_k$, then
{\rm (\ref{z})-(\ref{bdr})} has no nontrivial weak solution.
\end{Thm}
\begin{proof}
The proof is similar to the one of Theorem~\ref{curvau} so we just point out the
differences here.
The image of the half lines $\bar{b}=r\bar{a}$, with $\bar{a}$ increasing from zero
to $+\infty$, are curves in
the $(a,b)$-plane starting at
$(\lambda_{k+1}-\eps_3-1/\rho,\lambda_{k+1}-\eps_3-1/\rho)$ and ending at
$(\lambda_{k+1}-\eps_3,\lambda_{k+1}-\eps_3)$.
Consider the curve ${\cal D}_k$ parametrized by
$$
r\longmapsto
\left(\lambda_{k+1}-\eps_3-\frac{1}{\tilde{a}(r)+\rho},
\lambda_{k+1}-\eps_3-\frac{1}{r\tilde{a}(r)+\rho}\right).
$$
It passes through $(\lambda_k,\lambda_k)$ and divides the square $R_k:=
]\lambda_{k+1}-\eps_3-1/\rho,\lambda_{k+1}-\eps_3[^2
$ in two connected components.
We say that a point is inside $R_k$ and {\em above}\/ ${\cal D}_k$ if it belongs to
the component which
contains $]\lambda_k,\lambda_{k+1}-\eps_3[^2$.
There are no points on the Fu\v{c}ik spectrum of (\ref{z})-(\ref{bdr}) inside the square $R_k$
and above ${\cal D}_k$.
If $\lambda_k<0$ then
$R_k=\left]\lambda_{k+1}-\eps_3-1/\eps_4,\lambda_{k+1}-\eps_3\right[^2$,
and if $\lambda_k>0$ then
$$R_k=\left]
\eps_4\frac{(\lambda_{k+1}-\eps_3)^2}{1+\eps_4(\lambda_{k+1}-\eps_3)},
\lambda_{k+1}-\eps_3
\right[^2.$$

We now write $R_k=R_k(\eps_3,\eps_4)$.
If $\lambda_k<0$ we define $\tilde{R}_k
:=]-\infty,\lambda_{k+1}[^2$, and if $\lambda_k>0$ we define
$\tilde{R}_k
:=]0,\lambda_{k+1}[^2$. If $(a,b)\in\tilde{R}_k$ then we can choose
$\eps_3$, $\eps_4$ sufficiently small so that $(a,b)\in R_k(\eps_3,\eps_4)$.
\end{proof}

\section{An existence result for an asymptotically
linear problem}
We wish to prove the existence of a weak solution of
\begin{equation}\label{p}
\left\{\begin{array}{ll}
\Box u=au^+-bu^-+p&\mbox{in}\:]0,\pi[\times\T,\\
u(0,t)=u(\pi,t)=0&\mbox{for}\ 0\leq t\leq 2\pi.
\end{array}\right.
\end{equation}
As mentioned in the Introduction, our results should be compared with those of \cite{BFS},
\cite{McK5}, \cite{McK2}, \cite{McK1}, \cite{W}. An important difference is that we are able
to consider pairs $(a,b)$ outside the square $[\lambda_{k-1},\lambda_{k+1}]^2$.

Let us initially assume that $(a,b)$ lies in $\check{Q}_k$ and below ${\cal C}_k$.
Furthermore we assume $p:[0,\pi]\times\T\times\R\to\R$ is a continuous function, with
\begin{equation}\label{smallo}
 \lim_{|u|\to+\infty}\frac{p(s,t,u)}{u}=0\ \ \mbox{uniformly in $(s,t)$},
\end{equation}
satisfying:
\begin{itemize}
\item[(H1)] There exists $\underline{\eps}_0>0$ such that for each
$(s,t)\in[0,\pi]\times\T$ the function $u\mapsto
au^+-bu^-+p(s,t,u)-(\lambda_{k-1}+\underline{\eps}_0)u$ is
increasing.
 \item[(H2)] If $\lambda_k>0$ there exists
$\overline{\eps}_0>0$ such that for each $(s,t)\in[0,\pi]\times\T$
the function $u\mapsto p(s,t,u)-(1/\overline{\eps}_0)u$ is
decreasing.
 \item[(H3)] If $\lambda_k<0$ there exists
$\overline{\eps}_0>0$ such that for each $(s,t)\in[0,\pi]\times\T$
the function $u\mapsto au^+-bu^-+p(s,t,u)+\overline{\eps}_0u$ is
decreasing.
\end{itemize}
Choose $\eps_1$ and $\eps_2$ small enough so that $(a,b)\in
Q_k(\eps_1,\eps_2)=\,]p_k,q_k[^2$, with $Q_k$ given by (\ref{qk})
and $\mu$ given by (\ref{muu}). We recall that
$p_k\searrow\lambda_{k-1}$ as $\eps_1\searrow 0$; if $\lambda_k>0$
then $q_k\nearrow+\infty$ as $\eps_2\searrow 0$, and if
$\lambda_k<0$ then $q_k\nearrow 0$ as $\eps_2\searrow 0$.
By decreasing $\eps_1$ and $\eps_2$ if necessary
(so $p_k<\lambda_{k-1}+\underline{\eps}_0$, and
$q_k>(1/\overline{\eps}_0)+\max\{a,b\}$ if $\lambda_k>0$,
$q_k>-\overline{\eps}_0$ if $\lambda_k<0$),
we may assume that for each $(s,t)\in[0,\pi]\times\T$ the map
\begin{equation}\label{increase}
u\mapsto au^+-bu^-+p(s,t,u)-p_ku\ \ \mbox{is strictly increasing},
\end{equation}
and the map
\begin{equation}\label{decrease}
u\mapsto au^+-bu^-+p(s,t,u)-q_ku\ \  \mbox{is strictly decreasing}.
\end{equation}
Therefore, for each $(s,t)\in[0,\pi]\times\T$ the
function $p(s,t,\,\cdot\,)$ is absolutely continuous. In addition, it is true that
$p_k<b+p^\prime_u(s,t,\,\cdot\,)<q_k$ almost everywhere on $\R^-$,
and $p_k<a+p^\prime_u(s,t,\,\cdot\,)<q_k$ almost everywhere on
$\R^+$. Here $p^\prime_u(s,t,\,\cdot\,)$ denotes the derivative
with respect to the third variable. Conversely,
if these inequalities for the derivatives $p^\prime_u$ hold for some
$p_k$ and $q_k$ with $]p_k,q_k[^2=Q_k$,
then (H1)-(H3) hold. In other words,

\begin{Rem} {\rm (H1)-(H3)} hold if and only if
the derivative of the
nonlinear term of the differential equation with respect to $u$ is
between $p_k$ and $q_k$ almost everywhere for some
$p_k$ and $q_k$ with $]p_k,q_k[^2=Q_k$.
\end{Rem}

Define $P$, $J:[0,\pi]\times\T\times\R\to\R$ by
$P(s,t,u)=\int_0^u p(s,t,\tau)\,d\tau$ and
$$
 J(s,t,u)=H(u)+P(s,t,u),
$$
with $H$ given by (\ref{h}). From (\ref{increase}),
for each $(s,t)$ the function $J(s,t,\,\cdot\,)$ is convex.
Denote by $J^*(s,t,\,\cdot\,)$ the Fenchel-Legendre transform of $J(s,t,\,\cdot\,)$:
\begin{equation}\label{cara}
J^*(s,t,v)=\sup_{u\in\R}\,[vu-J(s,t,u)].
\end{equation}
Let $Q$ be such that
$$
 J^*(s,t,v)=H^*(v)+Q(s,t,v).
$$
and $q:=Q^\prime_v$.
Under our assumptions $Q(s,t,\,\cdot\,)$ is $C^1$
as $J(s,t,\,\cdot\,)$ is strictly convex and superlinear (see \cite[Proposition 2.4]{MW}).
Also, for any $v\in\R$, the map $(s,t)\mapsto J^*(s,t,v)$ is
measurable since the supremum in (\ref{cara}) can be taken over the set of rationals.
The function $J^*$ is Carathéodory.
\begin{Lem} Under assumption {\rm (\ref{smallo})},
we have
\begin{equation}\label{smallobis}
 \lim_{|v|\to+\infty}\frac{q(s,t,v)}{v}=0
\end{equation}
uniformly in $(s,t)$.
\end{Lem}
\begin{proof}
We denote by $\dot{a}=a-\lambda_{k-1}-\eps_1$ and $\dot{b}=b-\lambda_{k-1}-\eps_1$.
Let
$$
v=J^\prime_u(s,t,u)=\dot{a}u^+-\dot{b}u^-+p(s,t,u).
$$
We see that $|u|\to+\infty$ is equivalent to $|v|\to+\infty$.
Suppose $0<\eps<\min\left\{\frac{1}{\dot{a}},\frac{1}{\dot{b}}\right\}$.
We choose $c_1$ large enough so that for all $|u|>c_1$ we have
$$
\left|\frac{p(s,t,u)}{u}\right|<\frac{1}{2}\min\{\dot{a}^2,\dot{b}^2\}\eps.
$$
Take $c_2$ such that $|v|>c_2$ implies $|u|>c_1$.
In the first place, assume  $v>c_2$. Then
$$
v=v^+=\dot{a}u^+\left(1+\frac{p(s,t,u)}{\dot{a}u}\right),
$$
or
$$
u^+=\frac{1}{\dot{a}}v^+\frac{1}{1+p(s,t,u)/(\dot{a}u)}=\frac{1}{\dot{a}}v^+(1+y),
$$
where $|y|<\frac{2}{\dot{a}}\left|\frac{p(s,t,u)}{u}\right|<\dot{a}\eps$.
Indeed, we have used
$1/(1-x)=1+y$ with $|y|<2|x|$ for $|x|<1/2$.
On the other hand,
$$
u=(J^*)^\prime_v(s,t,v)
=\frac{1}{\dot{a}}v^+-\frac{1}{\dot{b}}v^-+q(s,t,v)=\frac{1}{\dot{a}}v^
+\left(1+\dot{a}\frac{q(s,t,v)}{v}\right).
$$
It follows that $y=\dot{a}\frac{q(s,t,v)}{v}$ and so
\begin{equation}\label{both}
\left|\frac{q(s,t,v)}{v}\right|<\eps.
\end{equation}
Similarly for $v<-c_2$. We have showed that $|v|>c_2$ implies
(\ref{both}). This proves the lemma.
\end{proof}
Clearly (\ref{smallobis}) yields
\begin{equation}\label{subqua}
\lim_{|v|\to+\infty}\frac{Q(s,t,v)}{v^2}=0
\end{equation}
uniformly in $(s,t)$.
Problem (\ref{p}) is equivalent to
\begin{equation}\label{pp}
 (L-\mu)v=(J^*)^\prime_v(\,\cdot\,,\,\cdot\,,v)-\mu v=\hat{a}v^+-\hat{b}v^-+q(\,\cdot\,,\,\cdot\,,v),
\end{equation}
with $v$ given by
$$
v=J^\prime_u(s,t,u).
$$
This is the Euler-Lagrange equation for the functional
$I:{\cal H}\to\R$ defined by
$$
I(v)=\frac{1}{2}\langle(L-\mu)v,v\rangle-\int\!\!\int_{[0,\pi]\times\T}\left[
J^*(s,t,v)-\frac{\mu}{2}|v|^2\right]ds\,dt.
$$

\begin{Prop}\label{arroz}
The functional $I$ has an absolute maximum.
\end{Prop}
\begin{proof}

Let $0<c<\min\{a-p_k,b-p_k\}$. There exists $d\geq 0$ such that $J(s,t,u)\geq
\frac{c}{2}u^2-d$ as the function $p$ is continuous. This leads to
$J^*(s,t,v)\leq\frac{1}{2c}v^2+d$. On the other hand
$J(s,t,0)=0$ implies $J^*(s,t,v)\geq 0$. So since $J^*$ is
Carathéodory, the functional $I$ is well defined.

By (\ref{increase}), the function $u\mapsto J(s,t,u)$ is twice
differentiable almost everywhere and its second distributional
derivative is nonnegative. Moreover, (\ref{decrease}) implies that
the function $u\mapsto
\frac{a}{2}(u^+)^2+\frac{b}{2}(u^-)^2+P(s,t,u)-\frac{q_k}{2}u^2
=J(s,t,u)-\frac{1}{2\mu}u^2$ is twice differentiable almost
everywhere and its second distributional derivative is
nonpositive. We conclude that the second distributional derivative
of $J(s,t,\,\cdot\,)$ coincides with its absolutely continuous
part. Furthermore, $J^{\prime\prime}_u(s,t,u)\leq 1/\mu$ almost
everywhere. On the other hand, since $J^*(s,t,\,\cdot\,)$ is
convex, its second distributional derivative is nonnegative. The
density of the absolutely continuous part of the second
distributional derivative of $J^*(s,t,\,\cdot\,)$ at
$J^\prime_u(s,t,u)$ is $(J^*)^{\prime\prime}_v(J^\prime_u(s,t,u))$
(derivative in the Alexandrov sense). At any point at which this
derivative exists $(J^*)^{\prime\prime}_v(J^\prime_u(s,t,u))=
[J^{\prime\prime}_u(s,t,u)]^{-1}\geq\mu$ (see \cite[pp.\
58-59]{V}). Therefore the second distributional derivative of the
map $v\mapsto J^*(s,t,v)-\frac{\mu}{2}|v|^2$ is nonnegative and
this map is convex.

For $\check{a}(r)$ as in (\ref{checka}) and all $v\in{\cal H}$, we know
$$
\langle(L-\mu)v,v\rangle-\check{a}(r)\|v^+\|^2-r\check{a}(r)\|v^-\|^2\leq 0.
$$
Let $(\hat{a},\hat{b})$ be the point given by (\ref{abt}).
Then
\begin{equation}\label{c}
\langle(L-\mu)v,v\rangle-\hat{a}\|v^+\|^2-\hat{b}\|v^-\|^2\leq -c\|v\|^2,
\end{equation}
with $c=\min\left\{\hat{a}-\check{a}\left(\frac{\hat{b}}{\hat{a}}\right),
\hat{b}-\check{a}\left(\frac{\hat{a}}{\hat{b}}\right)\right\}>0$.

Let $v_n$ be a maximizing sequence. Using (\ref{subqua}) and
(\ref{c}), we easily see that the sequence $(v_n)$ is bounded in
${\cal H}$. We write $v_n=w_n+z_n$ with $w_n\in{\cal R}$ and
$z_n\in{\cal N}$. Modulo a subsequence, we know $v_n\weak v_0$,
$w_n\weak w_0$ and $z_n\weak z_0$, all in ${\cal H}$. We prove
that $\limsup\|w_n\|=\|w\|$ and $\limsup\|z_n\|=\|z\|$. So suppose
that $\|w_n\|$ and $\|z_n\|$ converge. Since $Lw_n\to Lw_0$
strongly in ${\cal H}$, we have
\begin{eqnarray*}
 \sup I&\!\!\leq\!\!&I(v_0)-\mu\lim(\|w_n\|^2-\|w_0\|^2)
    +\nu\lim(\|z_n\|^2-\|z_0\|^2)\\
    &&\qquad\, \ -\left(\liminf{\int\!\!\int_{[0,\pi]\times\T}}\left[ J^*(s,t,v_n)-
    \frac{\mu}{2}|v_n|^2\right]ds\,dt\right.\\
        &&\left.\qquad\qquad\qquad\qquad\qquad-
    {\int\!\!\int_{[0,\pi]\times\T}}\left[ J^*(s,t,v_0)
    -\frac{\mu}{2}|v_0|^2\right]ds\,dt\right)
\end{eqnarray*}
The functional $v\mapsto \int\!\!\int_{[0,\pi]\times\T}\left[
J^*(s,t,v)-\frac{\mu}{2}|v|^2\right]ds\,dt$ is weakly
lower semi-con\-tin\-u\-ous in ${\cal H}$ because $v\mapsto
J^*(s,t,v)-\frac{\mu}{2}|v|^2$ is convex and bounded below (using
(\ref{subqua}), as $\hat{a},\hat{b}>0$). We conclude that
$$
I(v_0)=\sup I
$$
and $v_n\to v_0$ in ${\cal H}$.
\end{proof}
\begin{Thm}\label{nonu}
Let $p:[0,\pi]\times\T\times\R\to\R$ be a continuous function
satisfying\/ {\rm (\ref{smallo})} and\/ {\rm {(H1)}} to\/ {\rm
{(H3)}}. Suppose $(a,b)$ lies in $\check{Q}_k$ and below ${\cal
C}_k$. Then problem {\rm (\ref{p})} has a weak solution.
\end{Thm}
\begin{proof}
Any maximum point $v_0$ of $I$ is a solution of (\ref{pp}). The
equivalence between (\ref{p}) and (\ref{pp}) implies $u_0=Lv_0$ is
a weak solution of problem (\ref{p}).
\end{proof}

A similar result holds if $(a,b)$ lies in $\tilde{R}_k$ and above
${\cal D}_k$. In this case, (H1) to  (H3) above should be replaced
by
\begin{itemize}
\item[(H1$'$)] There exists $\underline{\eps}_0>0$ such that for
each $(s,t)\in[0,\pi]\times\T$ the function $u\mapsto
au^+-bu^-+p(s,t,u)-(\lambda_{k+1}-\underline{\eps}_0)u$ is
decreasing.
 \item[(H2$'$)] If $\lambda_k>0$ there exists
$\overline{\eps}_0>0$ such that for each $(s,t)\in[0,\pi]\times\T$
the function $u\mapsto au^+-bu^-+p(s,t,u)-\overline{\eps}_0u$ is
increasing.
 \item[(H3$'$)] If $\lambda_k<0$ there exists
$\overline{\eps}_0>0$ such that for each $(s,t)\in[0,\pi]\times\T$
the function $u\mapsto p(s,t,u)+(1/\overline{\eps}_0)u$ is
increasing.
\end{itemize}
\begin{Thm}\label{nond}
Let $p:[0,\pi]\times\T\times\R\to\R$ be a continuous function
satisfying\/ {\rm (\ref{smallo})} and\/ \/ {\rm (H1$'$)} to\/ {\rm
(H3$'$)}. Suppose $(a,b)$ lies in $\tilde{R}_k$ and above ${\cal
D}_k$. Then problem {\rm (\ref{p})} has a weak solution.
\end{Thm}

\end{document}